\documentclass[a4paper,12pt,reqno]{amsart}
\usepackage{amsfonts}
\usepackage{amsmath}
\usepackage{amssymb}
\usepackage[a4paper]{geometry}
\usepackage{mathrsfs}
\usepackage[colorlinks]{hyperref}
\usepackage{fullpage}
\usepackage{mathtools}
\usepackage{amsmath}
\numberwithin{equation}{section}
\theoremstyle{plain}
\newtheorem{theorem}{Theorem}[section]
\newtheorem{prop}[theorem]{Property}
\newtheorem{cor}[theorem]{Corollary}
\newtheorem{definition}[theorem]{Definition}

\newtheorem{remark}[theorem]{Remark}

\begin{document}

\title[Sobolev, Hardy, Gagliardo-Nirenberg and Caffarelli-Kohn-Nirenberg]{Sobolev, Hardy, Gagliardo-Nirenberg and Caffarelli-Kohn-Nirenberg type inequalities for some fractional derivatives}

\author[A. Kassymov]{Aidyn Kassymov}
\address{
   Aidyn Kassymov:
  \endgraf
  Department of Mathematics: Analysis, Logic and Discrete Mathematics
  \endgraf
  Ghent University, Belgium
  \endgraf
   and
  \endgraf
  Al-Farabi Kazakh National University
  \endgraf
  Almaty, Kazakhstan
  \endgraf
  and
  \endgraf
  Institute of Mathematics and Mathematical Modeling
  \endgraf
  Almaty, Kazakhstan
  \endgraf
  {\it E-mail address} {\rm kassymov@math.kz} and {\rm aidyn.kassymov@ugent.be}
}
  \author[M. Ruzhansky]{Michael Ruzhansky}
\address{
	Michael Ruzhansky:
	 \endgraf
  Department of Mathematics: Analysis, Logic and Discrete Mathematics
  \endgraf
  Ghent University, Belgium
  \endgraf
  and
  \endgraf
  School of Mathematical Sciences
    \endgraf
    Queen Mary University of London
  \endgraf
  United Kingdom
	\endgraf
  {\it E-mail address} {\rm michael.ruzhansky@ugent.be}
}

\author[N. Tokmagambetov]{Niyaz Tokmagambetov}
\address{
  Niyaz Tokmagambetov:
    \endgraf
  Department of Mathematics: Analysis, Logic and Discrete Mathematics
  \endgraf
  Ghent University, Belgium
  \endgraf
   and
  \endgraf
  Al-Farabi Kazakh National University
  \endgraf
  Almaty, Kazakhstan
  \endgraf
  {\it E-mail address} {\rm niyaz.tokmagambetov@ugent.be}
 }

\author[B. T. Torebek]{Berikbol T. Torebek}
\address{
  Berikbol T. Torebek:
    \endgraf
  Department of Mathematics: Analysis, Logic and Discrete Mathematics
  \endgraf
  Ghent University, Belgium
  \endgraf
   and
  \endgraf
  Institute of Mathematics and Mathematical Modeling
  \endgraf
  Almaty, Kazakhstan
  \endgraf
  {\it E-mail address} {\rm berikbol.torebek@ugent.be}
  }

\thanks{The authors were supported in parts by the FWO Odysseus 1 grant G.0H94.18N: Analysis and Partial Differential Equations. Michael Ruzhansky was supported in parts by the EPSRC Grant EP/R003025/1, by the Leverhulme Research Grant RPG-2017-151. Aidyn Kassymov was supported in parts by the MESRK Grant AP08053051 of the Ministry of Education and Science of the Republic of Kazakhstan}

\keywords{Sobolev inequality, Hardy inequality, Gagliardo-Nirenberg inequality,   Caffarelli-Kohn-Nirenberg inequality, fractional order differential operator, Caputo derivative, Riemann-Liouville derivative, Hadamard derivative.}
\subjclass[2010]{26D10, 45J05.}

\begin{abstract}
In this paper we show different inequalities for fractional order differential operators. In particular, the Sobolev, Hardy, Gagliardo-Nirenberg and Caffarelli-Kohn-Nirenberg type inequalities for the Caputo, Riemann-Liouville and Hadamard derivatives are obtained. In addition, we show some applications of these inequalities.
\end{abstract}

\maketitle

\tableofcontents

\section{Introduction}
There is no doubt that the inequalities not depending on a type of operators are very powerful for integral and differential equations. Without them, the progress of integro-differential equations would not be at its present level. Fractional order differential operators are not an exception.

Let us recall some classical results. Let $\Omega\subset\mathbb{R}^{N}$ be a measurable set and let $1<p<N$, then the classical Sobolev inequality is formulated as
\begin{equation}\label{cl-Sobolev}
\|u \|_{L^{p^*}(\Omega)}\leq C\|\nabla u \|_{L^{p}(\Omega)},\,\,\,u\in C^{\infty}_{0}(\Omega),
\end{equation}
where $C=C(N,p)>0$ is a positive constant, $p^{*}=\frac{Np}{N-p}$ and $ \nabla$ is the standard gradient in $\mathbb{R}^{N}$.  The inequality \eqref{cl-Sobolev} is one of the most important tools in PDEs and variational problems.

Further generalizations of the Sobolev inequality were obtained by Gagliardo and Nirenberg, independently.
In \cite{Gag} and \cite{Nir} they independently from each other proved the interpolation inequality
\begin{equation}\label{cl-GN}
 \|u\|^{p}_{L^{p}(\mathbb{R}^{N})}\leq C \|\nabla u\|^{N(p-2)/2}_{L^{2}(\mathbb{R}^{N})}\|u\|^{(2p-N(p-2))/2}_{L^{2}(\mathbb{R}^{N})},\,\,u\in H^{1}(\mathbb{R}^{N}),
\end{equation}
where \begin{equation*}
 \begin{cases}
   2\leq p\leq\infty\,\,\text{for}\,\, N = 2,\\
   2\leq p \leq \frac{2N}{N-2}\,\,\text{for} \,\,N > 2.
 \end{cases}
\end{equation*}
Now, it is called the Gagliardo--Nirenberg inequality.

The next important generalization of the Sobolev inequality is the Caffarelli--Kohn--Nirenberg inequality. In 1984, Caffarelli, Kohn and Nirenberg \cite{CKN} established the following result:
\begin{theorem}
Let $N\geq1$. Assume that $l_1$, $l_2$, $l_3$, $a, \, b, \, d,\, \delta \in \mathbb{R}$ be such that $l_1, l_2 \geq 1$,
$l_3 > 0, \,\,0 \leq \delta \leq 1,$ and
\begin{equation*}
\frac{1}{l_1}+\frac{a}{N},\,\,\,\frac{1}{l_2}+\frac{b}{N},\,\,\,\frac{1}{l_3}+\frac{\delta d+(1-\delta) b}{N}>0.
\end{equation*}
Then,
\begin{equation}\label{cl-CKN}
\||x|^{\delta d+(1-\delta) b}u\|_{L^{l_{3}}(\mathbb{R}^{N})}\leq C\||x|^{a}\nabla u\|^{\delta}_{L^{l_{1}}(\mathbb{R}^{N})}\||x|^{b} u\|^{1-\delta}_{L^{l_{2}}(\mathbb{R}^{N})},\,\,\,u\in C^{\infty}_{c}(\mathbb{R}^{N}),
\end{equation}
if and only if
\begin{multline*}
\frac{1}{l_3}+\frac{\delta d+(1-\delta) b}{N}=\delta\left(\frac{1}{l_{1}}+\frac{a-1}{N}\right)+
(1-\delta)\left(\frac{1}{l_2}+\frac{b}{N}\right),\\
a-d\geq0,\,\,\,\,\text{if}\,\,\,\delta>0,\\
a-d\leq1,\,\,\,\,\text{if}\,\,\,\delta>0\,\,\,\text{and}\,\,\,\frac{1}{l_3}+\frac{\delta d+(1-\delta) b}{N}=\frac{1}{l_1}+\frac{a-1}{N},
\end{multline*}
where $C$ is a positive constant independent of $u$.
\end{theorem}

Recently, mathematicians started to develop the classical inequalities \eqref{cl-Sobolev}, \eqref{cl-GN}, and \eqref{cl-CKN} for the $p$-Laplacian operator. In \cite{DPV}, Nezza, Palatucci and Valdinoci obtained the $p$-Laplacian version of the Sobolev inequality
\begin{equation}
\|u \|_{L^{p^*}(\mathbb{R}^{N})}\leq C [u]_{s,p},
\end{equation}
for the parameters $N>sp$, $1<p<\infty,$ and $s\in(0,1)$, for any measurable and compactly supported function $u$. Here, $C=C(N,p,s)>0$ is a suitable constant, and $[u]_{s,p}$ defined by
$$
[u]_{s,p}^{p}=\int_{\mathbb{R}^{N}}\int_{\mathbb{R}^{N}}
\frac{|u(x)-u(y)|^{p}}{|x-y|^{N+sp}}dxdy
$$
is the Gagliardo seminorm and $p^*=\frac{Np}{N-sp}$.

By using different techniques, the authors of the papers \cite{41, DLL, NS1} proved the Gagliardo-Nirenberg inequality for the $p$--Laplacian operator:
\begin{equation}
\|u\|_{L^{\tau}(\mathbb{R}^{N})}\leq C[u]^{a}_{s,p}\|u\|^{1-a}_{L^{\alpha}(\mathbb{R}^{N})},\,\,\forall u\in C^{1}_{c}(\mathbb{R}^{N}),
\end{equation}
 for $N\geq1,\,\,s\in(0,1),\,\,p>1,\,\,\alpha\geq1,\,\,\tau>0,$ and $a\in(0,1]$ such that
\begin{equation*}
\frac{1}{\tau}=a\left(\frac{1}{p}-\frac{s}{N}\right)+\frac{1-a}{\alpha}.
\end{equation*}

In \cite{27, 26} Hughes derived a Hardy-Landau-Littlewood
inequality \cite{24} for the Riemann-Liouville fractional integral, then for the Riemann-Liouville fractional
derivatives in weighted $L^p$ spaces. For more information about inequalities related to the fractional order operators, the reader is referred to \cite{AAK17} and references therein.

In this paper we deal with new inequalities related to some fractional order differential operators. Especially, the Caputo derivative analogues of the above inequalities are in the field of our interest. Here, we derive  the generalizations of the classical Sobolev, Hardy, Gagliardo-Nirenberg and Caffarelli-Kohn-Nirenberg inequalities. Note that in this direction systematic studies of different functional inequalities on general homogeneous (Lie) groups were initiated by the book
\cite{RSbook}.

Recently, more attention has been paid to the study of fractional analogues of known functional inequalities (see e.g. \cite{AAK17, An08, An09, An11, Iqbal}). Also, we note that in \cite{An09}, the author considered Sobolev-type inequality for the Caputo and Riemann-Liouville derivatives of order $\alpha\geq 1.$

We start by compiling basic definitions of fractional differential operators.

\section{Preliminaries}
Let us recall the Riemann--Liouville
fractional integrals and derivatives. Also, we give definitions of the Caputo fractional derivatives. In \cite[p.394]{1} the sequential differentiation was formulated in a way that we will use in the further investigations. We refer to \cite{samko, 1} and references therein for further properties.

\begin{definition} The left Riemann--Liouville
fractional integral $I_{a+} ^\alpha$ of order $\alpha>0$, and derivative $D_{a+} ^\alpha$ of order $0<\alpha\leq 1$ are given by
$$
I_{a+} ^\alpha  \left[ f \right]\left( t \right) = {\rm{
}}\frac{1}{{\Gamma \left( \alpha \right)}}\int\limits_a^t {\left(
{t - s} \right)^{\alpha  - 1} f\left( s \right)} ds, \,\,\, t\in(a,b],
$$
and
$$
D_{a+} ^\alpha \left[ f \right]\left( t \right) = \frac{{d }}{{dt
}}I_{a+} ^{1 - \alpha } \left[ f \right]\left( t \right), \,\,\, t\in(a, b],
$$
respectively and $f\in AC[a,b]$.  Here $\Gamma$ denotes the Euler gamma function.

Since $I^\alpha f(t) \rightarrow f(t)$ almost everywhere as $\alpha \rightarrow 0,$ then by definition we suppose that $I^0 f(t) = f(t).$ Hence $D_{a+}^1 f(t)=f'(t).$
\end{definition}

\begin{definition}
\label{fracderd}
The left Caputo fractional derivative of order $0<\alpha\leq 1$ is given by
$$
\partial_{a+} ^\alpha  \left[ f \right]\left( t \right) = D_{a+}
^\alpha  \left[ f\left( t \right) - f\left( a \right) \right]=I^{1-\alpha}_{a+}f'(t), \,\,\, t\in(a, b].
$$
\end{definition}

\begin{prop}
\label{rem1}
In Definition \ref{fracderd}, if $f(a)=0$, then $\partial_{a+} ^\alpha = D_{a+}^\alpha.$
\end{prop}

\begin{prop}
\label{prope1}
If $f\in L^1([a,b])$ and $\alpha>0,\, \beta>0,$ then the following equality holds $$I_{a+}^\alpha I_{a+}^\beta f(t)=I_{a+}^{\alpha+\beta}f(t).$$
\end{prop}

\begin{prop}[\cite{1}]
\label{prope2}
If $f\in L^1([a,b])$ and $f'\in L^1([a,b]),$ then the equality $$I_a^\alpha \partial_{a+}^\alpha f(t)=f(t)- f(a), \, 0<\alpha\leq 1,$$ holds almost everywhere on $[a,b].$\end{prop}

\section{The main results}

In this Section we derive the main results of this paper.

\begin{remark}
We note that in all statements of this section we will work with the Caputo fractional derivative $\partial^{\alpha}_{a+}$. But analogous results can be easily obtained for the Riemann-Liouville derivative $D_{a+} ^\alpha$ with the same order $\alpha\leq 1$ by adopting the techniques in the proofs and taking into account Property \ref{rem1}.
\end{remark}

\subsection{Poincar\'{e}--Sobolev type inequality}
In this subsection we show the Poincar\'{e}--Sobolev type inequality for fractional order operators.
\begin{theorem}\label{prop1cap} Let $u\in L^p(a,b),\,u(a)=0$, $\partial^{\alpha}_{a+}u\in L^p(a,b)$ and $p>1$. Then for the Caputo fractional derivative $\partial^{\alpha}_{a+}$ of order $\alpha\in \left(\frac{1}{p},1\right]$  we have the inequality
\begin{equation}\label{poinineq}
\|u\|_{L^{\infty}(a,b)}\leq \frac{(b-a)^{\alpha-\frac{1}{p}}}{\left(\frac{\alpha p}{p-1}-\frac{1}{p-1}\right)^\frac{p-1}{p}\Gamma(\alpha)} \left\|\partial^{\alpha}_{a+}u\right\|_{L^p(a,b)}.
\end{equation}
\end{theorem}
\begin{proof} Let $u\in L^p(a,b),\,u(a)=0$, $\partial^{\alpha}_{a+}u\in L^p(a,b)$ and consider the function
\begin{equation}
u(t)=I^{\alpha}_{a+}\partial^{\alpha}_{a+}u(t).
\end{equation}
Using the H\"{o}lder inequality with $\frac{1}{p}+\frac{1}{q}=1$, we obtain
\begin{align*}
\left|I^{\alpha}_{a+}\partial^{\alpha}_{a+}u(t)\right|&\leq \frac{1}{\Gamma(\alpha)}\int\limits_a^t\left|(t-s)^{\alpha-1}\partial^{\alpha}_{a+}u(s)\right|ds\\&\leq \frac{1}{\Gamma(\alpha)}\left(\int\limits_a^t(t-s)^{\alpha q-q} ds \right)^{\frac{1}{q}} \left(\int\limits_a^t \left|\partial^{\alpha}_{a+}u(s)\right|^p ds\right)^{\frac{1}{p}}\\&
\stackrel{\alpha>\frac{1}{p}}=\frac{(t-a)^{\alpha-1+\frac{1}{q}}}{\left(\alpha q-q+1\right)^\frac{1}{q}\Gamma(\alpha)} \left(\int\limits_a^t \left|\partial^{\alpha}_{a+}u(s)\right|^p ds\right)^{\frac{1}{p}}\\&\leq \frac{(b-a)^{\alpha-1+\frac{1}{q}}}{\left(\alpha q-q+1\right)^\frac{1}{q}\Gamma(\alpha)} \left\|\partial^{\alpha}_{a+}u\right\|_{L^p(a,b)}\\&
=\frac{(b-a)^{\alpha-\frac{1}{p}}}{\left(\alpha q-q+1\right)^\frac{1}{q}\Gamma(\alpha)} \left\|\partial^{\alpha}_{a+}u\right\|_{L^p(a,b)}\\&
=\frac{(b-a)^{\alpha-\frac{1}{p}}}{\left(\frac{\alpha p}{p-1}-\frac{1}{p-1}\right)^\frac{p-1}{p}\Gamma(\alpha)} \left\|\partial^{\alpha}_{a+}u\right\|_{L^p(a,b)},
\end{align*}
where $q=\frac{p}{p-1}>1$.

Then,
\begin{equation}
    \|u\|_{L^{\infty}(a,b)}=\|I^{\alpha}_{a+}\partial^{\alpha}_{a+}u\|_{L^{\infty}(a,b)}\leq \frac{(b-a)^{\alpha-\frac{1}{p}}}{\left(\frac{\alpha p}{p-1}-\frac{1}{p-1}\right)^\frac{p-1}{p}\Gamma(\alpha)} \left\|\partial^{\alpha}_{a+}u\right\|_{L^p(a,b)},
\end{equation}
showing \eqref{poinineq}.
\end{proof}
\begin{remark}
In Theorem \ref{prop1cap}, by taking $1<q<\infty$, we obtain
\begin{equation}
    \|u\|_{L^{q}(a,b)}\leq \frac{(b-a)^{\alpha-\frac{1}{p}+\frac{1}{q}}}{\left(\frac{\alpha p}{p-1}-\frac{1}{p-1}\right)^\frac{p-1}{p}\Gamma(\alpha)} \left\|\partial^{\alpha}_{a+}u\right\|_{L^p(a,b)}.
\end{equation}
\end{remark}
Let us also present the following  result.

\begin{theorem} \label{sobolev type-1}
Let $\partial^{\alpha}_{a+}u\in L^p(a,b)$ with $p>1$ and let $\beta\in [0,1)$ be such that $\alpha\in\left(\beta+\frac{1}{p},1\right]$. Then for the Caputo fractional derivative $\partial^{\beta}_{a+}$, we have
\begin{equation}\label{soblevt1}
\|\partial^{\beta}_{a+}u\|_{L^{\infty}(a,b)}\leq \frac{(b-a)^{\alpha-\beta-\frac{1}{p}+\frac{1}{q}}}{\left(\alpha q-\beta q-q+1\right)^\frac{1}{q}\Gamma(\alpha-\beta)} \left\|\partial^{\alpha}_{a+}u\right\|_{L^p(a,b)},
\end{equation}
for all $1<  p\leq q<\infty$, where $\frac{1}{p}+\frac{1}{q}=1$.
\end{theorem}

\begin{proof} By using Definition \ref{fracderd} and Properties \ref{prope1} and \ref{prope2}, we introduce the function
\begin{equation}\partial^{\beta}_{a+}u(t)=I^{1-\beta}_{a+}u'(t)=I^{\alpha-\beta}_{a+}I^{1-\alpha}_{a+}u'(t)=I^{\alpha-\beta}_{a+}\partial^{\alpha}_{a+}u(t).\end{equation}
Using the H\"{o}lder inequality with $\frac{1}{p}+\frac{1}{q}=1$, we get
\begin{align*}
\left|I^{\alpha-\beta}_{a+}\partial^{\alpha}_{a+}u(t)\right|&\leq \frac{1}{\Gamma(\alpha-\beta)}\int\limits_a^t\left|(t-s)^{\alpha-\beta-1}\partial^{\alpha}_{a+}u(s)\right|ds\\&\leq \frac{1}{\Gamma(\alpha-\beta)}\left(\int\limits_a^t(t-s)^{\alpha q-\beta q-q} ds \right)^{\frac{1}{q}} \left(\int\limits_a^t \left|\partial^{\alpha}_{a+}u(s)\right|^p ds\right)^{\frac{1}{p}}\\&=\frac{(t-a)^{\alpha-\beta-1+\frac{1}{q}}}{\left(\alpha q-\beta q-q+1\right)^\frac{1}{q}\Gamma(\alpha-\beta)} \left(\int\limits_a^t \left|\partial^{\alpha}_{a+}u(s)\right|^p ds\right)^{\frac{1}{p}}
\\&=\frac{(t-a)^{\alpha-\beta-\frac{1}{p}}}{\left(\alpha q-\beta q-q+1\right)^\frac{1}{q}\Gamma(\alpha-\beta)} \left(\int\limits_a^t \left|\partial^{\alpha}_{a+}u(s)\right|^p ds\right)^{\frac{1}{p}}
\\&\leq \frac{(b-a)^{\alpha-\beta-\frac{1}{p}}}{\left(\alpha q-\beta q-q+1\right)^\frac{1}{q}\Gamma(\alpha-\beta)} \left\|\partial^{\alpha}_{a+}u\right\|_{L^p(a,b)}
,
\end{align*}
where by assumption $\alpha>\beta+\frac{1}{p}$, we have $\alpha q-\beta q-q+1>0$.
From this, we obtain
\begin{align}
\|\partial^{\beta}_{a+}u\|_{L^{\infty}(a,b)}&\leq \frac{(b-a)^{\alpha-\beta-\frac{1}{p}}}{\left(\alpha q-\beta q-q+1\right)^\frac{1}{q}\Gamma(\alpha-\beta)} \left\|\partial^{\alpha}_{a+}u\right\|_{L^p(a,b)},
\end{align}
showing \eqref{soblevt1}.
\end{proof}

\begin{remark}
In \eqref{soblevt1}, if $\beta=0$, we obtain the Sobolev type inequality.
\end{remark}

\begin{remark}
In Theorem \ref{sobolev type-1}, by taking $1<q<\infty$, we get
\begin{align}
\|\partial^{\beta}_{a+}u\|_{L^{q}(a,b)}&\leq \frac{(b-a)^{\alpha-\beta-\frac{1}{p}+\frac{1}{q}}}{\left(\alpha q-\beta q-q+1\right)^\frac{1}{q}\Gamma(\alpha-\beta)} \left\|\partial^{\alpha}_{a+}u\right\|_{L^p(a,b)}.
\end{align}

\end{remark}

\subsection{Hardy type inequality} Let us show the Hardy inequality.
\begin{theorem}\label{thmharcap}
Let $a>0,\,\,u(a)=0$ and $\partial^{\alpha}_{a+}u\in L^p(a,b)$ with $p>1$ and $\alpha\in \left(\frac{1}{p},1\right]$. Then for the Caputo fractional derivative $\partial^{\alpha}_{a+}$ we have the inequality
\begin{equation}\label{hardycap}
\left\|\frac{u}{x}\right\|_{L^p(a,b)}\leq \frac{a^{-1}(b-a)^{\alpha}}{\left(\frac{\alpha p}{p-1}-\frac{1}{p-1}\right)^\frac{p-1}{p}\Gamma(\alpha)} \left\|\partial^{\alpha}_{a+}u\right\|_{L^p(a,b)}.
\end{equation}

\end{theorem}
\begin{proof}
From $a< x< b$ we have $\frac{1}{b}<\frac{1}{x}<\frac{1}{a}.$ By using Theorem \ref{prop1cap}, we calculate
\begin{equation}\label{hhh}
\begin{split}
\left(\int_{a}^{b}\frac{|u(x)|^{p}}{x^{p}}dx\right)^{\frac{1}{p}}&=\left(\int_{a}^{b}x^{-p}|u(x)|^{p}dx\right)^{\frac{1}{p}}\\&
\leq a^{-1}\|u\|_{L^{p}(a,b)}\\&
\stackrel{\eqref{poinineq}}\leq\frac{a^{-1}(b-a)^{\alpha}}{\left(\frac{\alpha p}{p-1}-\frac{1}{p-1}\right)^\frac{p-1}{p}\Gamma(\alpha)} \left\|\partial^{\alpha}_{a+}u\right\|_{L^p(a,b)},
\end{split}
\end{equation}
showing \eqref{hardycap}.
\end{proof}
Let us give the weighted one-dimensional Hardy type inequality.

\begin{theorem}\label{weiHar}
Let $a>0$, $u\in L^p(a,b),\,u(a)=0$ and $\partial^{\alpha}_{a+}u\in L^p(a,b)$ with $p>1$ and $\alpha\in \left(\frac{1}{p},1\right]$. Then for the Caputo fractional derivative $\partial^{\alpha}_{a+}$ of order $\alpha$ and $\gamma\in\mathbb{R}$
we have
\begin{equation}\label{weighthar}
\left\|\frac{u}{x^{\gamma+1}}\right\|_{L^p(a,b)}\leq \frac{a^{-|\gamma|-1}b^{|\gamma|}(b-a)^{\alpha}}{\left(\alpha q-q+1\right)^\frac{1}{q}\Gamma(\alpha)} \left\|\frac{\partial^{\alpha}_{a+}u}{x^{\gamma}}\right\|_{L^p(a,b)},
\end{equation}
where $q=\frac{p}{p-1}.$
\end{theorem}
\begin{proof}
We prove our statement in two stages, namely, when $\gamma\geq0$ and $\gamma<0$. Firstly, let us study the case $\gamma\geq0$. For $a>0$, we have $b^{-\gamma-1}<x^{-\gamma-1}<a^{-\gamma-1}$, so that
\begin{equation}
 \begin{split}
\left(\int_{a}^{b}\frac{|u(x)|^{p}}{x^{(\gamma+1)p}}dx\right)^{\frac{1}{p}}&\leq a^{-\gamma-1}\left(\int_{a}^{b}|u(x)|^{p}dx\right)^{\frac{1}{p}}\\&
\stackrel{\eqref{poinineq}}\leq \frac{a^{-\gamma-1}(b-a)^{\alpha}}{\left(\frac{\alpha p}{p-1}-\frac{1}{p-1}\right)^\frac{p-1}{p}\Gamma(\alpha)} \left(\int_{a}^{b}|\partial^{\alpha}_{a+}u|^{p}dx\right)^{\frac{1}{p}}\\&
=\frac{a^{-\gamma-1}(b-a)^{\alpha}}{\left(\frac{\alpha p}{p-1}-\frac{1}{p-1}\right)^\frac{p-1}{p}\Gamma(\alpha)} \left(\int_{a}^{b}\frac{x^{\gamma p}}{x^{\gamma p}}|\partial^{\alpha}_{a+}u|^{p}dx\right)^{\frac{1}{p}}\\&
\leq\frac{a^{-\gamma-1}b^{\gamma }(b-a)^{\alpha}}{\left(\frac{\alpha p}{p-1}-\frac{1}{p-1}\right)^\frac{p-1}{p}\Gamma(\alpha)} \left(\int_{a}^{b}\frac{|\partial^{\alpha}_{a+}u|^{p}}{x^{\gamma p}}dx\right)^{\frac{1}{p}}\\&
=\frac{a^{-\gamma-1}b^{\gamma}(b-a)^{\alpha}}{\left(\alpha q-q+1\right)^\frac{1}{q}\Gamma(\alpha)} \left\|\frac{\partial^{\alpha}_{a+}u}{x^{\gamma}}\right\|_{L^p(a,b)}.
\end{split}
\end{equation}

To show the case $\gamma<0$, one obtains
\begin{equation}
 \begin{split}
\left(\int_{a}^{b}\frac{|u(x)|^{p}}{x^{(\gamma+1)p}}dx\right)^{\frac{1}{p}}&=\left(\int_{a}^{b}\frac{|u(x)|^{p}}{x^{(\gamma p+p)}}dx\right)^{\frac{1}{p}}
\\&\leq b^{-\gamma}\left(\int_{a}^{b}\frac{|u(x)|^{p}}{x^{p}}dx\right)^{\frac{1}{p}}\\&
\stackrel{\eqref{hardycap}}\leq \frac{a^{-1}b^{-\gamma}(b-a)^{\alpha}}{\left(\frac{\alpha p}{p-1}-\frac{1}{p-1}\right)^\frac{p-1}{p}\Gamma(\alpha)} \left\|\partial^{\alpha}_{a+}u\right\|_{L^p(a,b)}\\&
= \frac{a^{-1}b^{-\gamma}(b-a)^{\alpha}}{\left(\frac{\alpha p}{p-1}-\frac{1}{p-1}\right)^\frac{p-1}{p}\Gamma(\alpha)} \left(\int_{a}^{b}|\partial^{\alpha}_{a+}u|^{p}dx\right)^{\frac{1}{p}}\\&
=\frac{a^{-1}b^{-\gamma}(b-a)^{\alpha}}{\left(\frac{\alpha p}{p-1}-\frac{1}{p-1}\right)^\frac{p-1}{p}\Gamma(\alpha)} \left(\int_{a}^{b}\frac{x^{\gamma p}}{x^{\gamma p}}|\partial^{\alpha}_{a+}u|^{p}dx\right)^{\frac{1}{p}}\\&
\leq \frac{a^{\gamma-1}b^{-\gamma}(b-a)^{\alpha}}{\left(\frac{\alpha p}{p-1}-\frac{1}{p-1}\right)^\frac{p-1}{p}\Gamma(\alpha)} \left(\int_{a}^{b}\frac{|\partial^{\alpha}_{a+}u|^{p}}{x^{\gamma p}}dx\right)^{\frac{1}{p}}\\&
=\frac{a^{\gamma-1}b^{-\gamma}(b-a)^{\alpha}}{\left(\frac{\alpha p}{p-1}-\frac{1}{p-1}\right)^\frac{p-1}{p}\Gamma(\alpha)} \left\|\frac{\partial^{\alpha}_{a+}u}{x^{\gamma}}\right\|_{L^{p}(a,b)},
\end{split}
\end{equation}
implying \eqref{weighthar}.
\end{proof}
\subsection{Gagliardo-Nirenberg type inequality}
Now, we are on a way to establish the Gagliardo-Nirenberg inequality for differential operators of fractional orders. We show that the Sobolev type inequality formulated in Theorem \ref{sobolev type-1} implies a family of Gagliardo--Nirenberg inequalities.

\begin{theorem}\label{LM: GN-Ineq-y}
Assume that $1\leq p, q<\infty$, $\alpha\in\left(\frac{1}{q},1\right]$ and $u(a)=0$.
Then we have the following Gagliardo-Nirenberg type inequality
\begin{equation}\label{EQ: GN-inequalitycap}
\|u\|_{L^{\gamma}(a, b)}\leq  C\|\partial^{\alpha}_{a+} u\|_{L^{q}(a, b)}^{s} \|u\|_{L^{p}(a, b)}^{1-s},
\end{equation}
with
\begin{equation}
\frac{\gamma s}{q}+\frac{\gamma(1-s)}{p}=1,
\end{equation}
where $s\in[0,1]$.
\end{theorem}

\begin{proof} 
By using the H\"{o}lder inequality with $\frac{\gamma s}{q}+\frac{\gamma(1-s)}{p}=1,$  we have
\begin{equation}
\begin{split}
    \int_{a}^{b}|u(x)|^{\gamma}dx&=\int_{a}^{b}|u(x)|^{\gamma s}|u(x)|^{\gamma(1-s)}dx\\&
    \leq\left(\int_{a}^{b}|u(x)|^{q}dx\right)^{\frac{\gamma s}{q}} \left(\int_{a}^{b}|u(x)|^{p}dx\right)^{\frac{\gamma(1-s)}{p}}\\&
    \stackrel{\eqref{poinineq}}\leq C\|\partial^{\alpha}_{a+} u\|^{\gamma s}_{L^{q}(a, b)} \|u\|_{L^{p}(a, b)}^{\gamma (1-s)},
\end{split}
\end{equation}
showing \eqref{EQ: GN-inequalitycap}.
\end{proof}

Let us consider the space $\dot{H}_{+}^{\alpha}(a,b)$ with $\alpha\in\left(\frac{1}{2},1\right]$ of the following form
\begin{equation*}
    \dot{H}_{+}^{\alpha}(a,b):=\{u\in L^{2}(a,b),\,\,\partial_{a+}^{\alpha}u\in L^{2}(a,b),\,\,u(a)=0\}.
\end{equation*}
In particular case of Theorem \ref{LM: GN-Ineq-y}, which is important for our further analysis, when $q=2$ and $\alpha=1$, one obtains the classical Gagliardo-Nirenberg inequality:

\begin{cor}\label{CR: GN-Ineq-s on graded Lie Groups}
We have the following Gagliardo-Nirenberg type inequality
\begin{equation}
\label{EQ: GN-inequality on graded Lie Groups}
\|u\|_{L^{\gamma}(a, b)}\leq
C \|u\|_{\dot{H}_{+}^{1}(a, b)}^{s} \|u\|_{L^{p}(a, b)}^{1-s},
\end{equation}
for $s\in[0, 1]$.
\end{cor}

We also recall another more general special case of Theorem \ref{LM: GN-Ineq-y} with $q=2$:
\begin{cor}
Let $\alpha\in\left(\frac{1}{2},1\right]$. Assume also that $1\leq p<\infty$ and $s\in[0,1].$ Then we have the following Gagliardo-Nirenberg type inequality
\begin{equation}
\label{EQ: GN-inequality on graded Lie Groups}
\|u\|_{L^{\gamma}(a, b)}\lesssim \|u\|_{\dot{H}_{+}^{\alpha}(a, b)}^{s} \|u\|_{L^{p}(a, b)}^{1-s},
\end{equation}
for $\frac{1}{\gamma}=\frac{s}{2}+\frac{1-s}{p}$.
\end{cor}

\subsection{Caffarelli-Kohn-Nirenberg type inequality} Now let us show the fractional Caffarelli-Kohn-Nirenberg type inequality.

\begin{theorem}\label{CKN}
 Assume that $a>0$, $\alpha\in\left(1-\frac{1}{q},1\right)$, $1<p,q<\infty$, $0<r<\infty,$ and $p+q\geq r$. Let $\delta\in [0,1]\cap[\frac{r-q}{r},\frac{p}{r}]$ and $c,d,e\in \mathbb{R}$ with  $\frac{\delta}{p}+\frac{1-\delta}{q}=\frac{1}{r}$,
$c=\delta(d-1)+e(1-\delta)$ and $u(a)=0$. If $1+(d-1)p>0$ then we have
\begin{equation}\label{CKNpqr}
\|x^{c}u\|_{L^{r}(a,b)}\leq C\|x^{d}\partial^{\alpha}_{a+}u\|^{\delta}_{L^{p}(a,b)}\|x^{e}u\|^{1-\delta}_{L^{q}(a,b)}.
\end{equation}
\end{theorem}
\begin{proof}
Case $\delta=0$.

If $\delta=0,$ then $c=e$ and $q=r$. Then  \eqref{CKNpqr} is the inequality
$$\|x^{c}u\|_{L^{r}(a,b)}\leq \|x^{c}u\|_{L^{r}(a,b)}.$$

Case $\delta=1$.

If $\delta=1$, then we have $c=d-1$ and $p=r$. Also, we have $1+cp=1+(d-1)p>0$. Then by using weighted fractional Hardy inequality (Theorem \ref{weiHar}), we obtain
\begin{equation}
\begin{split}
 \left\|x^{c}u\right\|_{L^{p}(a,b)}&\leq C\left\|x^{c+1}\partial^{\alpha}_{a+}u\right\|_{L^{p}(a,b)}
 \\&= C\left\|x^{d}\partial^{\alpha}_{a+}u\right\|_{L^{p}(a,b)}.
\end{split}
\end{equation}

Case $\delta\in [0,1]\cap[\frac{r-q}{r},\frac{p}{r}]$.

By assumption $c=\delta(d-1)+e(1-\delta)$ and by using H\"{o}lder's inequality with $\frac{\delta}{p}+\frac{1-\delta}{q}=\frac{1}{r}$, we calculate
\begin{equation}
\begin{split}
\|x^{c}u\|_{L^{r}(a,b)}&=\left(\int_{a}^{b}x^{cr}|u(x)|^{r}dx\right)^{\frac{1}{r}}
\\&=\left(\int_{a}^{b}\frac{|u(x)|^{\delta r}}{x^{\delta r(1-d)}}\frac{|u(x)|^{(1-\delta) r}}{x^{-er(1-\delta)}}dx\right)^{\frac{1}{r}}
\\&\leq\left\|\frac{u}{x^{1-d}}\right\|^{\delta}_{L^{p}(a,b)}\left\|\frac{u}{x^{-e}}\right\|^{1-\delta}_{L^{q}(a,b)}.
\end{split}
\end{equation}
By using weighted fractional Hardy inequality (Theorem \ref{weiHar}) with $1+(d-1)p>0$, we obtain
\begin{equation}
\begin{split}
\|x^{c}u\|_{L^{r}(a,b)}&\leq \left\|\frac{u}{x^{1-d}}\right\|^{\delta}_{L^{p}(a,b)}\left\|\frac{u}{x^{-e}}\right\|^{1-\delta}_{L^{q}(a,b)}
\\&\leq C\|x^{d}\partial^{\alpha}_{a+}u\|^{\delta}_{L^{p}(a,b)}\|x^{e}u\|^{1-\delta}_{L^{q}(a,b)},
\end{split}
\end{equation}
completing the proof.
\end{proof}


\section{Sequential Derivation Case}

In this subsection we collect results for the sequential derivatives. Indeed, these results are important due to the non--commutativity and the absence of the semi--group property of fractional differential operators.
\subsection{Fractional Poincare--Sobolev type inequality}
\begin{theorem}\label{spropsob} Let $\partial^{\beta}_{a+}u(a)=0$, $\partial^{\alpha}_{a+}\partial^{\beta}_{a+}u\in L^p(a,b)$
with  $\alpha \in\left(\frac{1}{q},1\right)$ and $\beta\in (0,1).$ Then the following inequality is true
\begin{equation}\label{se1}
\|\partial^{\beta}_{a+}u\|_{L^{\infty}(a,b)}\leq \frac{(b-a)^{\alpha-\frac{1}{p}}}{\left(\alpha q-q+1\right)^\frac{1}{q}\Gamma(\alpha)} \left\|\partial^{\alpha}_{a+}\partial^{\beta}_{a+}u\right\|_{L^p(a,b)},
\end{equation}
with $\frac{1}{p}+\frac{1}{q}=1$.
\end{theorem}
\begin{proof} Consider the function
\begin{equation}\partial^{\beta}_{a+}u(t)=I^{\alpha}_{a+}\partial^{\alpha}_{a+}\partial^{\beta}_{a+}u(t).\end{equation}
Using the H\"{o}lder inequality, one has
\begin{align*}
\left|I^{\alpha}_{a+}\partial^{\alpha}_{a+}\partial^{\beta}_{a+}u(t)\right|&\leq \frac{1}{\Gamma(\alpha)}\int\limits_a^t\left|(t-s)^{\alpha-1}\partial^{\alpha}_{a+}\partial^{\beta}_{a+}u(s)\right|ds\\&\leq \frac{1}{\Gamma(\alpha)}\left(\int\limits_a^t(t-s)^{\alpha q-q} ds \right)^{\frac{1}{q}} \left(\int\limits_a^t \left|\partial^{\alpha}_{a+}\partial^{\beta}_{a+}u(s)\right|^p ds\right)^{\frac{1}{p}}\\&=\frac{(t-a)^{\alpha-1+\frac{1}{q}}}{\left(\alpha q-q+1\right)^\frac{1}{q}\Gamma(\alpha)} \left(\int\limits_a^t \left|\partial^{\alpha}_{a+}\partial^{\beta}_{a+}u(s)\right|^p ds\right)^{\frac{1}{p}}\\&\leq \frac{(b-a)^{\alpha-1+\frac{1}{q}}}{\left(\alpha q-q+1\right)^\frac{1}{q}\Gamma(\alpha)} \left\|\partial^{\alpha}_{a+}\partial^{\beta}_{a+}u\right\|_{L^p(a,b)}.\end{align*}
Then we obtain
\begin{align*}
\|\partial^{\beta}_{a+}u\|_{L^{\infty}(a,b)}&\leq \frac{(b-a)^{\alpha-\frac{1}{p}}}{\left(\alpha q-q+1\right)^\frac{1}{q}\Gamma(\alpha)} \left\|\partial^{\alpha}_{a+}\partial^{\beta}_{a+}u\right\|_{L^p(a,b)},
\end{align*}
completing proof.
\end{proof}
\begin{remark}
If $1<\theta<\infty$ in Theorem \ref{spropsob}, then we have
\begin{align*}
\|\partial^{\beta}_{a+}u\|_{L^{\theta}(a,b)}&\leq \frac{(b-a)^{\alpha-\frac{1}{p}+\frac{1}{\theta}}}{\left(\alpha q-q+1\right)^\frac{1}{q}\Gamma(\alpha)} \left\|\partial^{\alpha}_{a+}\partial^{\beta}_{a+}u\right\|_{L^p(a,b)}.
\end{align*}
\end{remark}
\subsection{Fractional Hardy type inequality} Now we show the following sequential fractional Hardy inequality.
\begin{theorem} Let $a>0$, $\gamma\in\mathbb{R}$ and $\partial^{\beta}_{a+}u(a)=0$ and $\partial^{\alpha}_{a+}\partial^{\beta}_{a+}u\in L^p(a,b)$
with $\alpha \in\left(\frac{1}{q},1\right).$  Then the following inequality is true
\begin{equation}\label{e1seq}
\left\|\frac{\partial^{\beta}_{a+}u}{x}\right\|_{L^p(a,b)}\leq C \left\|\partial^{\alpha}_{a+}\partial^{\beta}_{a+}u\right\|_{L^p(a,b)},
\end{equation}
with $\frac{1}{p}+\frac{1}{q}=1$.
\end{theorem}
\begin{proof}
From $a< x< b$ we have $\frac{1}{b}<\frac{1}{x}<\frac{1}{a}.$ By using Theorem \ref{spropsob}, we calculate
\begin{equation}\label{hhh}
\begin{split}
\left(\int_{a}^{b}\frac{|\partial^{\beta}_{a+} u(x)|^{p}}{x^{p}}dx\right)^{\frac{1}{p}}&=\left(\int_{a}^{b}x^{-p}|\partial^{\beta}_{a+}u(x)|^{p}dx\right)^{\frac{1}{p}}\\&
\leq a^{-1}\|\partial^{\beta}_{a+} u\|_{L^{p}(a,b)}\\&
\stackrel{\eqref{se1}}\leq\frac{a^{-1}(b-a)^{\alpha}}{\left(\alpha q-q+1\right)^\frac{1}{q}\Gamma(\alpha)} \left\|\partial^{\alpha}_{a+} \partial^{\beta}_{a+}u\right\|_{L^p(a,b)},
\end{split}
\end{equation}
showing \eqref{e1seq}.
\end{proof}

\subsection{Fractional Gagliardo-Nirenberg type inequality} In the same way as Theorem \ref{LM: GN-Ineq-y} is proved, we can prove the following statement.

\begin{theorem}\label{TH-seq-GN}
Assume that $1\leq p, q<\infty$, and let $\alpha\in(0,1)$ be such that $\beta\in\left(\frac{1}{q},1\right)$. Suppose that $\partial^{\alpha}_{a+}\partial^{\beta}_{a+}u\in L^{q}(a,b)$ and $\partial^{\alpha}_{a}u\in L^{p}(a,b)$.
Then we have the following Gagliardo-Nirenberg type inequality
\begin{equation}
\label{EQ: GN-seq-inequality-2}
\int_{a}^{b}|\partial^{\alpha}_{a+}u(x)|^{\gamma}dx\lesssim \left(\int_{a}^{b}|\partial^{\beta}_{a+}\partial^{\alpha}_{a+}u(x)|^{q}dx\right)^{\frac{s\gamma}{q}} \left(\int_{a}^{b}|\partial^{\alpha}_{a+}u(x)|^{p}dx\right)^{\frac{(1-s)\gamma}{p}},
\end{equation}
with
\begin{equation}
\frac{s\gamma}{q}+\frac{(1-s)\gamma}{p}=1,
\end{equation}
where $s\in[0,1].$
\end{theorem}
\begin{proof}
Let us calculate the following integral
\begin{equation}
\begin{split}
\int_{a}^{b}|\partial^{\alpha}_{a+}u(x)|^{\gamma}dx&=\int_{a}^{b}|\partial^{\alpha}_{a+}u(x)|^{s\gamma}|\partial^{\alpha}_{a+}u(x)|^{(1-s)\gamma}dx
\\&\leq \left(\int_{a}^{b}|\partial^{\alpha}_{a+}u(x)|^{q}dx\right)^{\frac{s\gamma}{q}}\left(\int_{a}^{b}|\partial^{\alpha}_{a+}u(x)|^{p}dx\right)^{\frac{(1-s)\gamma}{p}},
\end{split}
\end{equation}
with
\begin{equation}
\frac{s\gamma}{q}+\frac{(1-s)\gamma}{p}=1.
\end{equation}
Then by using Theorem \ref{spropsob}, we obtain
\begin{equation}
\begin{split}
\int_{a}^{b}|\partial^{\alpha}_{a+}u(x)|^{\gamma}dx&
\leq\left(\int_{a}^{b}|\partial^{\alpha}_{a+}u(x)|^{q}dx\right)^{\frac{s\gamma}{q}}\left(\int_{a}^{b}|\partial^{\alpha}_{a+}u(x)|^{p}dx\right)^{\frac{(1-s)\gamma}{p}}\\&
\stackrel{\eqref{e1seq}}\leq C \left(\int_{a}^{b}|\partial^{\beta}_{a+}\partial^{\alpha}_{a+}u(x)|^{q}dx\right)^{\frac{s\gamma}{q}} \left(\int_{a}^{b}|\partial^{\alpha}_{a+}u(x)|^{p}dx\right)^{\frac{(1-s)\gamma}{p}}.
\end{split}
\end{equation}
The theorem is proved.
\end{proof}
\section{Hadamard fractional derivative}
Let us give the definition of the Hadamard fractional derivative.

\begin{definition} The left  Hadamard
fractional integral $\mathfrak{I}_{a+} ^\alpha$ of order $\alpha>0$, and derivative $\mathfrak{D}_{a+} ^\alpha$  of order $0<\alpha<1$ are given by
$$\mathfrak{I}_{a+} ^\alpha  \left[ f \right]\left(t\right)=\frac{1}{\Gamma \left(\alpha \right)}\int_{a}^{t} \left(\log\frac{t}{s}\right)^{\alpha-1}  f\left( s \right) \frac{ds}{s}, \,\,\, t\in(a,b],$$
and
$$\mathfrak{D}_{a+}^{\alpha} \left[ f \right]\left( t \right) = \frac{1}{\Gamma(1-\alpha)}\int_{a}^{t}\left(\log \frac{t}{s}\right)^{-\alpha}f'\left( s \right) \frac{ds}{s}. \,\,\, t\in(a, b].$$
 Here $\Gamma$ denotes the Euler gamma function.
\end{definition}

\begin{prop}[\cite{1}]\label{prope2had} If $f\in L^1(a,b)$ and $f'\in L_{\frac{1}{x}}^1(a,b),$ then the equality $$\mathfrak{I}_a^\alpha \mathfrak{D}_{a+} ^\alpha f(t)=f(t)- f(a), \, 0<\alpha<1,$$ holds almost everywhere on $[a,b].$\end{prop}

Now for $p\geq1$ we define the weighted Lebesgue space $L^{p}_{\frac{1}{x}}(a,b)$ with the norm
\begin{equation}
    \|u\|_{L^{p}_{\frac{1}{x}}(a,b)}:=\left(\int_{a}^{b}|u(x)|^{p}\frac{dx}{x}\right)^{\frac{1}{p}}.
\end{equation}

For our further purpose we will need the following property of the weighted space $L^{p}_{\frac{1}{x}}(a,b)$.
\begin{prop}[\cite{1}]\label{prope1had} Suppose that $f\in L_{\frac{1}{x}}^1(a,b)$. Then for the parameters $\alpha>0$ and $\beta>0$ we have the following equality
$$
\mathfrak{I}_{a+}^\alpha \mathfrak{I}_{a+}^\beta f(t)= \mathfrak{I}_{a+}^{\alpha+\beta}f(t),
$$
for almost all $t\in(a, b)$.
\end{prop}

\subsection{Poincar\'{e}--Sobolev type inequality}

In this subsection we show the fractional order Poincar\'{e}--Sobolev type inequality.

\begin{theorem}\label{prop1}
Let $a>0$  and $p>1$. Assume that $u\in L^p(a,b)$ and $\mathfrak{D}^{\alpha}_{a+}u\in L_{\frac{1}{x}}^p(a,b)$ with $u(a)=0$. Then for the Hadamard fractional derivative $\mathfrak{D}^{\alpha}_{a+}$ of order $\alpha\in \left(\frac{1}{p},1\right]$  we have
\begin{equation}\label{poinineqhad}
\|u\|_{L^{\infty}(a,b)}\leq \frac{\left|\log\frac{b}{a}\right|^{\alpha-\frac{1}{p}}}{\left(\frac{\alpha p}{p-1}-\frac{1}{p-1}\right)^\frac{p-1}{p}\Gamma(\alpha)} \left\|\mathfrak{D}^{\alpha}_{a+}u\right\|_{L_{\frac{1}{x}}^p(a,b)}.
\end{equation}
\end{theorem}
\begin{proof} Let $u\in L_{\frac{1}{x}}^p(a,b),\,u(a)=0$, $\mathfrak{D}^{\alpha}_{a+}u\in L^p(a,b)$ and consider the function
\begin{equation}
u(t)=\mathfrak{I}^{\alpha}_{a+}\mathfrak{D}^{\alpha}_{a+}u(t).
\end{equation}
Using the H\"{o}lder inequality with $\frac{1}{p}+\frac{1}{q}=1$, we obtain
\begin{align*}
\left|\mathfrak{I}^{\alpha}_{a+}\mathfrak{D}^{\alpha}_{a+}u(t)\right|&\leq \frac{1}{\Gamma(\alpha)}\int\limits_a^t\left|\left(\log \frac{t}{s}\right)^{\alpha-1}\mathfrak{D}^{\alpha}_{a+}u(s)\right|\frac{ds}{s^{\frac{1}{p}+\frac{1}{q}}}\\&
\leq \frac{1}{\Gamma(\alpha)}\left(\int\limits_a^t\left|\log \frac{t}{s}\right|^{\alpha q-q} \frac{ds}{s} \right)^{\frac{1}{q}} \left(\int\limits_a^t \left|\mathfrak{D}^{\alpha}_{a+}u(s)\right|^p \frac{ds}{s}\right)^{\frac{1}{p}}\\&
\stackrel{\alpha>\frac{1}{p}}=\frac{\left|\log\frac{t}{a}\right|^{\alpha-1+\frac{1}{q}}}{\left(\alpha q-q+1\right)^\frac{1}{q}\Gamma(\alpha)} \left(\int\limits_a^t \left|\mathfrak{D}^{\alpha}_{a+}u(s)\right|^p \frac{ds}{s}\right)^{\frac{1}{p}}\\&
\leq \frac{\left|\log\frac{b}{a}\right|^{\alpha-1+\frac{1}{q}}}{\left(\alpha q-q+1\right)^\frac{1}{q}\Gamma(\alpha)} \left\|\mathfrak{D}^{\alpha}_{a+}u\right\|_{L_{\frac{1}{x}}^p(a,b)}\\&
=\frac{\left|\log\frac{b}{a}\right|^{\alpha-\frac{1}{p}}}{\left(\alpha q-q+1\right)^\frac{1}{q}\Gamma(\alpha)} \left\|\mathfrak{D}^{\alpha}_{a+}u\right\|_{L_{\frac{1}{x}}^p(a,b)}\\&
=\frac{\left|\log\frac{b}{a}\right|^{\alpha-\frac{1}{p}}}{\left(\frac{\alpha p}{p-1}-\frac{1}{p-1}\right)^\frac{p-1}{p}\Gamma(\alpha)} \left\|\mathfrak{D}^{\alpha}_{a+}u\right\|_{L_{\frac{1}{x}}^p(a,b)},
\end{align*}
where $q=\frac{p}{p-1}>1$, showing \eqref{poinineqhad}.

\end{proof}
\begin{remark}
In Theorem \ref{prop1}, by taking $1<\theta<\infty,$ we have
\begin{equation}
\|u\|_{L^{\theta}(a,b)}\leq \frac{(b-a)^{\frac{1}{\theta}}\left|\log\frac{b}{a}\right|^{\alpha-\frac{1}{p}}}{\left(\frac{\alpha p}{p-1}-\frac{1}{p-1}\right)^\frac{p-1}{p}\Gamma(\alpha)} \left\|\mathfrak{D}^{\alpha}_{a+}u\right\|_{L_{\frac{1}{x}}^p(a,b)}.
\end{equation}
\end{remark}

\subsection{Hardy type inequality}
Here, we show the Hardy inequality for the Hadamard derivative.

\begin{theorem}\label{thmharhad}
Let $a>0$ and $p>1$. Assume that $\mathfrak{D}^{\alpha}_{a+}u\in L_{\frac{1}{x}}^p(a,b)$ and $u(a)=0$. Then for the Hadamard fractional derivative $\mathfrak{D}^{\alpha}_{a+}$ of order $\alpha\in \left(\frac{1}{p},1\right]$ we have
\begin{equation}\label{hardyhad}
\left\|\frac{u}{x}\right\|_{L^p(a,b)}\leq \frac{a^{-1}(b-a)^{\frac{1}{p}}\left|\log\frac{b}{a}\right|^{\alpha-\frac{1}{p}}}{\left(\frac{\alpha p}{p-1}-\frac{1}{p-1}\right)^\frac{p-1}{p}\Gamma(\alpha)} \left\|\mathfrak{D}^{\alpha}_{a+}u\right\|_{L^{p}_{\frac{1}{x}}(a,b)}.
\end{equation}

\end{theorem}
\begin{proof}
From $a< x< b$ we have $\frac{1}{b}<\frac{1}{x}<\frac{1}{a}.$ By using Theorem \ref{prop1}, we calculate
\begin{equation}\label{hhhhad}
\begin{split}
\left(\int_{a}^{b}\frac{|u(x)|^{p}}{x^{p}}dx\right)^{\frac{1}{p}}&=\left(\int_{a}^{b}x^{-p}|u(x)|^{p}dx\right)^{\frac{1}{p}}\\&
\leq a^{-1}\|u\|_{L^{p}(a,b)}\\&
\stackrel{\eqref{poinineqhad}}\leq\frac{a^{-1}(b-a)^{\frac{1}{p}}\left|\log\frac{b}{a}\right|^{\alpha-\frac{1}{p}}}{\left(\frac{\alpha p}{p-1}-\frac{1}{p-1}\right)^\frac{p-1}{p}\Gamma(\alpha)} \left\|\mathfrak{D}^{\alpha}_{a+}u\right\|_{L^{p}_{\frac{1}{x}}(a,b)},
\end{split}
\end{equation}
showing \eqref{hardyhad}.
\end{proof}

Let us show the weighted Hardy inequality with the Hadamard derivative.
\begin{theorem}\label{weiHarhad}
Let $a>0,\,\,u(a)=0$ and $\mathfrak{D}^{\alpha}_{a+}u\in L_{\frac{1}{x}}^p(a,b)$ with $p>1$. Then for the Hadamard fractional derivative $\mathfrak{D}^{\alpha}_{a+}$ of order $\alpha\in \left(\frac{1}{p},1\right]$ and $\gamma\in\mathbb{R}$,
 we have inequality
\begin{equation}\label{weightharhad}
\left\|\frac{u}{x^{\gamma+1}}\right\|_{L^p(a,b)}\leq C \left\|\frac{\mathfrak{D}^{\alpha}_{a+}u}{x^{\gamma}}\right\|_{L_{\frac{1}{x}}^p(a,b)}.
\end{equation}
\end{theorem}
\begin{proof}
We prove this result in two steps. Let us first show the case $\gamma\geq0$.  Since $a>0$ we have $b^{-\gamma-1}<x^{-\gamma-1}<a^{-\gamma-1}$, and by the direct calculations one obtains
\begin{equation}
 \begin{split}
\left(\int_{a}^{b}\frac{|u(x)|^{p}}{x^{(\gamma+1)p}}dx\right)^{\frac{1}{p}}&\leq a^{-\gamma-1}\left(\int_{a}^{b}|u(x)|^{p}dx\right)^{\frac{1}{p}}\\&
\stackrel{\eqref{poinineqhad}}\leq \frac{a^{-\gamma-1}(b-a)^{\frac{1}{p}}\left|\log\frac{b}{a}\right|^{\alpha-\frac{1}{p}}}{\left(\frac{\alpha p}{p-1}-\frac{1}{p-1}\right)^\frac{p-1}{p}\Gamma(\alpha)} \left(\int_{a}^{b}|\mathfrak{D}^{\alpha}_{a+}u|^{p}\frac{dx}{x}\right)^{\frac{1}{p}}\\&
=\frac{a^{-\gamma-1}(b-a)^{\frac{1}{p}}\left|\log\frac{b}{a}\right|^{\alpha-\frac{1}{p}}}{\left(\frac{\alpha p}{p-1}-\frac{1}{p-1}\right)^\frac{p-1}{p}\Gamma(\alpha)} \left(\int_{a}^{b}\frac{x^{\gamma p}}{x^{\gamma p}}|\mathfrak{D}_{a+}^{\alpha}u|^{p}\frac{dx}{x}\right)^{\frac{1}{p}}\\&
\leq\frac{a^{-\gamma-1}(b-a)^{\frac{1}{p}}b^{\gamma }\left|\log\frac{b}{a}\right|^{\alpha-\frac{1}{p}}}{\left(\frac{\alpha p}{p-1}-\frac{1}{p-1}\right)^\frac{p-1}{p}\Gamma(\alpha)} \left(\int_{a}^{b}\frac{|\mathfrak{D}^{\alpha}_{a+}u|^{p}}{x^{\gamma p}}\frac{dx}{x}\right)^{\frac{1}{p}}\\&
=\frac{a^{-\gamma-1}(b-a)^{\frac{1}{p}}b^{\gamma }\left|\log\frac{b}{a}\right|^{\alpha-\frac{1}{p}}}{\left(\frac{\alpha p}{p-1}-\frac{1}{p-1}\right)^\frac{p-1}{p}\Gamma(\alpha)} \left\|\frac{\mathfrak{D}_{a+}^{\alpha}u}{x^{\gamma}}\right\|_{L_{\frac{1}{x}}^p(a,b)}.
\end{split}
\end{equation}

Now to prove the case $\gamma<0$ we arrive at
\begin{equation}
 \begin{split}
\left(\int_{a}^{b}\frac{|u(x)|^{p}}{x^{(\gamma+1)p}}dx\right)^{\frac{1}{p}}&=\left(\int_{a}^{b}\frac{|u(x)|^{p}}{x^{(\gamma p+p)}}dx\right)^{\frac{1}{p}}
\\&\leq b^{-\gamma}\left(\int_{a}^{b}\frac{|u(x)|^{p}}{x^{p}}dx\right)^{\frac{1}{p}}\\&
\stackrel{\eqref{hardyhad}}\leq \frac{b^{-\gamma}a^{-1}(b-a)^{\frac{1}{p}}\left|\log\frac{b}{a}\right|^{\alpha-\frac{1}{p}}}{\left(\frac{\alpha p}{p-1}-\frac{1}{p-1}\right)^\frac{p-1}{p}\Gamma(\alpha)} \left\|\mathfrak{D}^{\alpha}_{a+}u\right\|_{L^{p}_{\frac{1}{x}}(a,b)}\\&
= \frac{b^{-\gamma}a^{-1}(b-a)^{\frac{1}{p}}\left|\log\frac{b}{a}\right|^{\alpha-\frac{1}{p}}}{\left(\frac{\alpha p}{p-1}-\frac{1}{p-1}\right)^\frac{p-1}{p}\Gamma(\alpha)} \left(\int_{a}^{b}|\mathfrak{D}^{\alpha}_{a+}u|^{p}\frac{dx}{x}\right)^{\frac{1}{p}}\\&
=\frac{b^{-\gamma}a^{-1}(b-a)^{\frac{1}{p}}\left|\log\frac{b}{a}\right|^{\alpha-\frac{1}{p}}}{\left(\frac{\alpha p}{p-1}-\frac{1}{p-1}\right)^\frac{p-1}{p}\Gamma(\alpha)} \left(\int_{a}^{b}\frac{x^{\gamma p}}{x^{\gamma p}}|\mathfrak{D}^{\alpha}_{a+}u|^{p}\frac{dx}{x}\right)^{\frac{1}{p}}\\&
\leq \frac{b^{-\gamma}a^{\gamma-1}(b-a)^{\frac{1}{p}}\left|\log\frac{b}{a}\right|^{\alpha-\frac{1}{p}}}{\left(\frac{\alpha p}{p-1}-\frac{1}{p-1}\right)^\frac{p-1}{p}\Gamma(\alpha)} \left(\int_{a}^{b}\frac{|\mathfrak{D}^{\alpha}_{a+}u|^{p}}{x^{\gamma p}}\frac{dx}{x}\right)^{\frac{1}{p}}\\&
=\frac{b^{-\gamma}a^{\gamma-1}(b-a)^{\frac{1}{p}}\left|\log\frac{b}{a}\right|^{\alpha-\frac{1}{p}}}{\left(\frac{\alpha p}{p-1}-\frac{1}{p-1}\right)^\frac{p-1}{p}\Gamma(\alpha)} \left\|\frac{\mathfrak{D}^{\alpha}_{a+}u}{x^{\gamma}}\right\|_{L_{\frac{1}{x}}^{p}(a,b)},
\end{split}
\end{equation}
showing \eqref{weightharhad}.
\end{proof}

\subsection{Fractional Gagliardo-Nirenberg type inequality with  Hadamard derivative}
\begin{theorem}\label{LM: GN-Ineq-yhad}
Assume that $\alpha\in\left(\frac{1}{q},1\right]$, $1\leq p, q<\infty$.
Then we have the following Gagliardo-Nirenberg type inequality
\begin{equation}
\label{EQ: GN-inequality}
\|u\|_{L^{\gamma}(a, b)}\leq  C\|\mathfrak{D}^{\alpha}_{a+} u\|_{L_{\frac{1}{x}}^{q}(a, b)}^{s} \|u\|_{L^{p}(a, b)}^{1-s},
\end{equation}
with
\begin{equation}
\frac{\gamma s}{q}+\frac{\gamma(1-s)}{p}=1,
\end{equation}
where $s\in[0,1]$.
\end{theorem}

\begin{proof} 
By using the H\"{o}lder inequality with $\frac{\gamma s}{q}+\frac{\gamma(1-s)}{p}=1,$  we get
\begin{equation}
\begin{split}
    \int_{a}^{b}|u(x)|^{\gamma}dx&=\int_{a}^{b}|u(x)|^{\gamma s}|u|^{\gamma(1-s)}dx\\&
    \leq\left(\int_{a}^{b}|u(x)|^{q}dx\right)^{\frac{\gamma s}{q}} \left(\int_{a}^{b}|u(x)|^{p}dx\right)^{\frac{\gamma(1-s)}{p}}\\&
    \stackrel{\eqref{poinineqhad}}\leq C\|\mathfrak{D}^{\alpha}_{a+} u\|^{\gamma s}_{L_{\frac{1}{x}}^{q}(a, b)} \|u\|_{L^{p}(a, b)}^{\gamma (1-s)},
\end{split}
\end{equation}
completing the proof.
\end{proof}

\subsection{Fractional Caffarelli-Kohn-Nirenberg type inequality with Hadamard derivative}

Now we are in a position to show the fractional Cafarrelli-Kohn-Nirenberg type inequality.
\begin{theorem}
\label{CKN}
Let $a>0$, $1<p,q<\infty$, $\alpha\in\left(1-\frac{1}{q},1\right)$, and $0<r<\infty$ such that $p+q\geq r$. Suppose that $\delta\in [0,1]\cap[\frac{r-q}{r},\frac{p}{r}]$ and $c,d,e\in \mathbb{R}$ with $\frac{\delta}{p}+\frac{1-\delta}{q}=\frac{1}{r}$ and
$c=\delta(d-1)+e(1-\delta)$. Assume that  $x^{d}\mathfrak{D}^{\alpha}_{a+}u\in L_{\frac{1}{x}}^{p}(a,b)$, $x^{e}u\in L^{q}(a,b)$ and $u(a)=0$.

Moreover, let $1+(d-1)p>0$ then we have $x^{c}u\in L^{r}(a,b)$ and
\begin{equation}\label{CKNpqrhad}
\|x^{c}u\|_{L^{r}(a,b)}\leq C\left\|x^{d}\mathfrak{D}^{\alpha}_{a+}u\right\|_{L_{\frac{1}{x}}^{p}(a,b)}^{\delta}\|x^{e}u\|^{1-\delta}_{L^{q}(a,b)}.
\end{equation}
\end{theorem}
\begin{proof}
Case $\delta=0$.

If $\delta=0,$ then $c=e$ and $q=r$. Then  \eqref{CKNpqr} is the inequality
$$\|x^{c}u\|_{L^{r}(a,b)}\leq \|x^{c}u\|_{L^{r}(a,b)}.$$

Case $\delta=1$.

If $\delta=1$, then we have $c=d-1$ and $p=r$. Also, we have $1+cp=1+(d-1)p>0$. Then by using weighted fractional Hardy inequality (Theorem \ref{weiHarhad}) we obtain
\begin{equation}
\begin{split}
 \left\|x^{c}u\right\|_{L^{p}(a,b)}&\leq C\left\|x^{c+1}\mathfrak{D}^{\alpha}_{a+}u\right\|_{L_{\frac{1}{x}}^{p}(a,b)}
 \\&= C\left\|x^{d}\mathfrak{D}^{\alpha}_{a+}u\right\|_{L_{\frac{1}{x}}^{p}(a,b)}.
\end{split}
\end{equation}

Case $\delta\in [0,1]\cap[\frac{r-q}{r},\frac{p}{r}]$.

By assuming $c=\delta(d-1)+e(1-\delta)$ and using the H\"{o}lder's inequality with $\frac{\delta}{p}+\frac{1-\delta}{q}=\frac{1}{r}$, we calculate
\begin{equation}
\begin{split}
\|x^{c}u\|_{L^{r}(a,b)}&=\left(\int_{a}^{b}x^{cr}|u(x)|^{r}dx\right)^{\frac{1}{r}}
\\&=\left(\int_{a}^{b}\frac{|u(x)|^{\delta r}}{x^{\delta r(1-d)}}\frac{|u(x)|^{(1-\delta) r}}{x^{-er(1-\delta)}}dx\right)^{\frac{1}{r}}
\\&\leq\left\|\frac{u}{x^{1-d}}\right\|^{\delta}_{L^{p}(a,b)}\left\|\frac{u}{x^{-e}}\right\|^{1-\delta}_{L^{q}(a,b)}.
\end{split}
\end{equation}
By using the weighted fractional Hardy inequality (Theorem \ref{weiHarhad}) with $1+(d-1)p>0$, we obtain
\begin{equation}
\begin{split}
\|x^{c}u\|_{L^{r}(a,b)}&\leq \left\|\frac{u}{x^{1-d}}\right\|^{\delta}_{L^{p}(a,b)}\left\|\frac{u}{x^{-e}}\right\|^{1-\delta}_{L^{q}(a,b)}
\\&\leq C\|x^{d}\mathfrak{D}^{\alpha}_{a+}u\|^{\delta}_{L_{\frac{1}{x}}^{p}(a,b)}\|x^{e}u\|^{1-\delta}_{L^{q}(a,b)},
\end{split}
\end{equation}
showing \eqref{CKNpqrhad}.
\end{proof}

\section{Applications}
In this Section we show some applications of the obtained inequalities for the real-valued functions $u$.

\subsection{Uncertainly principle}
The inequality \eqref{hardycap} implies the following uncertainly principle:
\begin{cor}\label{uncercap}
Let $a>0$, $u(a)=0$ and $\partial^{\alpha}_{a+}u\in L^p(a,b)$ with $p>1$.  Then for the Caputo fractional derivative $\partial^{\alpha}_{a+}$ of order $\alpha\in \left(\frac{1}{p},1\right]$ we have following inequality
\begin{equation}\label{up}
\left\|u\right\|^{2}_{L^2(a,b)}\leq \frac{a^{-1}(b-a)^{\alpha}}{\left(\frac{\alpha p}{p-1}-\frac{1}{p-1}\right)^\frac{p-1}{p}\Gamma(\alpha)} \left\|\partial^{\alpha}_{a+}u\right\|_{L^p(a,b)}\|xu\|_{L^q(a,b)},
\end{equation}
where $q=\frac{p}{p-1}$.
\end{cor}
\begin{proof}
By using \eqref{hardycap}, we obtain
\begin{equation}
    \begin{split}
        \frac{a^{-1}(b-a)^{\alpha}}{\left(\frac{\alpha p}{p-1}-\frac{1}{p-1}\right)^\frac{p-1}{p}\Gamma(\alpha)} \left\|\partial^{\alpha}_{a+}u\right\|_{L^p(a,b)}\|xu\|_{L^q(a,b)}&\stackrel{\eqref{hardycap}}\geq\left\|\frac{u}{x}\right\|_{L^p(a,b)}\|xu\|_{L^q(a,b)}\\&
        \geq \left\|u\right\|^{2}_{L^2(a,b)},
    \end{split}
\end{equation}
completing the proof.
\end{proof}
\begin{remark}
Also, the uncertainly principle holds for the Riemann-Liouville derivative.
\end{remark}
Let us show uncertainly principle for the Hadamard derivative.
\begin{cor}
Let $a>0$ and $p>1$. Assume that $\mathfrak{D}^{\alpha}_{a+}u\in L_{\frac{1}{x}}^p(a,b)$ and $u(a)=0$. Then for the Hadamard fractional derivative $\mathfrak{D}^{\alpha}_{a+}$ of order $\alpha\in \left(\frac{1}{p},1\right]$ we have
\begin{equation}
\|u\|^{2}_{L^{2}(a,b)}\leq  \frac{a^{-1}(b-a)^{\frac{1}{p}}\left|\log\frac{b}{a}\right|^{\alpha-\frac{1}{p}}}{\left(\frac{\alpha p}{p-1}-\frac{1}{p-1}\right)^\frac{p-1}{p}\Gamma(\alpha)} \left\|\mathfrak{D}^{\alpha}_{a+}u\right\|_{L^{p}_{\frac{1}{x}}(a,b)}\|xu\|_{L^{q}(a,b)},
\end{equation}
where $q=\frac{p}{p-1}$.
\end{cor}
\begin{proof}
Proof is similar to Corollary \ref{uncercap} with using Theorem \ref{thmharhad}.
\end{proof}
\subsection{Embedding of spaces}
Let us consider the space $\dot{H}^{\alpha}_{+}(a,b)$ with $\alpha\in\left(\frac{1}{2},1\right]$ introduced in \cite{Baj01, GLY15} in the following form
\begin{equation*}
    \dot{H}^{\alpha}_{+}(a,b):=\{u\in L^{2}(a,b),\,\,\partial_{a+}^{\alpha}u\in L^{2}(a,b),\,\,u(a)=0\}.
\end{equation*}
If $\alpha<\beta$, then by the Poincar\'{e}--Sobolev-type inequality \eqref{poinineq} we have $\dot{H}^{\beta}_{+}(a,b)\hookrightarrow \dot{H}^{\alpha}_{+}(a,b).$

Let us introduce the space $\dot{\mathcal{W}}^{\alpha}_{+}$ in the following form
  \begin{equation*}
  \dot{\mathcal{W}}^{\alpha}_{+}(a,b):=\{u\in L^{2}(a,b),\,\,\mathfrak{D}^{\alpha}_{a+}u\in L^{2}(a,b),\,\,u(a)=0\},
  \end{equation*}
  where $\mathfrak{D}^{\alpha}_{a+}$ is the left Hadamard derivative. If $\alpha<\beta$, then by the Poincar\'{e}--Sobolev-type inequality \eqref{poinineqhad} we have $\dot{\mathcal{W}}^{\beta}_{+}(a,b)\hookrightarrow \dot{\mathcal{W}}^{\alpha}_{+}(a,b).$
\subsection{A-priori estimate}
Here, we seek a real-valued solution to the following space-fractional diffusion problem
\begin{equation}\label{heat}
\begin{cases}
u_{t}(x,t)+D^{\alpha}_{b-}\partial_{a+}^{\alpha}u(x,t)=0,\,\,\,\,(x,t)\in (a,b)\times(0,T),\\
u(x,0)=u_{0}(x),\,\,\,\forall x\in(a,b),\\
 \end{cases}
\end{equation}
where $\alpha\in\left(\frac{1}{2},1\right]$, $u\in L^{\infty}(0,T;\dot{H}^{\alpha}_{+}(a,b))$, $u_{t}\in L^{2}(0,T;\dot{H}^{\alpha}_{+}(a,b))$ and $u_{0}\in L^{2}(a,b)$.

Now we show an a-priori estimate for this problem. Let us define
$$I(t)=\|u(x,\cdot)\|^{2}_{L^{2}(a,b)}=\int_{a}^{b}|u(x,t)|^{2}dx.$$
Then by multiplying \eqref{heat} by $u$, integrating over $(a,b)$,  and by using integration by parts, we compute
\begin{equation}\label{energy}
    \begin{split}
        \int_{a}^{b}u_{t}(x,t)u(x,t)dx&+\int_{a}^{b}u(x,t)D^{\alpha}_{b-}\partial_{a+}^{\alpha}u(x,t)dx\\&=\frac{1}{2}\frac{d}{dt}\int_{a}^{b}|u(x,t)|^{2}dx+\int_{a}^{b}|\partial_{a+}^{\alpha}u(x,t)|^{2}dx\\&
        =\frac{1}{2}\frac{dI(t)}{dt}+\int_{a}^{b}|\partial_{a+}^{\alpha}u(x,t)|^{2}dx.
    \end{split}
\end{equation}

By using \eqref{poinineq} with $p=2$ in \eqref{energy}, we get
\begin{equation*}
    \begin{split}
0=&\frac{1}{2}\frac{dI(t)}{dt}+\int_{a}^{b}|\partial_{a+}^{\alpha}u(x,t)|^{2}dx\\&\stackrel{\eqref{poinineq}}\geq\frac{1}{2}\frac{dI(t)}{dt}+\frac{\left(2\alpha -1\right)\Gamma^{2}(\alpha)}{(b-a)^{2\alpha}}\int_{a}^{b}|u(x,t)|^{2}dx.
\end{split}
\end{equation*}
Consequently, we arrive at $\frac{dI(t)}{dt}\leq 0$. This means that $I(t)$ is a non-decreasing function. Then for all $t>0$ we have
$I(t)\leq I(0)$. Thus,
$$
\|u(x,\cdot)\|_{L^{2}(a,b)}\leq \|u_{0}\|_{L^{2}(a,b)}.
$$

\end{document}